\documentclass{gtpart}
\pdfoutput=1
\agtart
\usepackage{PageSetup,Definitions}
\usepackage{tikz-cd}
\hypersetup{colorlinks=false,hidelinks,pdfborder={0 0 0},backref=false}

\title{The $T$-algebra spectral sequence: comparisons and applications}
\author{Justin Noel}                                               
\givenname{Justin}
\surname{Noel}
\address{University of Regensburg\\NWF I - Mathematik\\D-93040 Regensburg\\Germany}
\email{justin.noel@mathematik.uni-regensburg.de}
\urladdr{http://www.nullplug.org/}


\volumenumber{}
\issuenumber{}
\publicationyear{}
\papernumber{}
\startpage{}
\endpage{}
\doi{}
\MR{}
\Zbl{}
\received{}
\revised{}
\accepted{}
\published{}
\publishedonline{}
\proposed{}
\seconded{}
\corresponding{}
\editor{}
\version{}

\usepackage{graphicx}
\numberwithin{equation}{section} 
\numberwithin{figure}{section} 

\theoremstyle{plain}
\newtheorem{prop}{Proposition}[section]
\newtheorem{thm}{Theorem}[section]           

\theoremstyle{definition}
\newtheorem{defn}{Definition}[section]
\newtheorem{remark}{Remark}[section]

\makeatletter\let\c@equation\c@thm\makeatother
\makeatletter\let\c@figure\c@thm\makeatother
\makeatletter\let\c@prop\c@thm\makeatother
\makeatletter\let\c@defn\c@thm\makeatother
\makeatletter\let\c@remark\c@thm\makeatother
\newcommand{\hkmod}{\HkMod}
\newcommand{\shkmod}{\sHkMod}

\newcommand{\fp}{\mathbb{F}_p}
\newcommand{\fpbar}{\overline{\mathbb{F}}_p}
\newcommand{\phat}{\widehat{p}}

\begin{document}
\begin{abstract}
 In previous work with Niles Johnson the author constructed a spectral sequence for computing homotopy groups of spaces of maps between structured objects such as $G$--spaces and $\mathcal{E}_n$--ring spectra. In this paper we study special cases of this spectral sequence in detail. Under certain assumptions, we show that the Goerss--Hopkins spectral sequence and the $T$--algebra spectral sequence agree. Under further assumptions, we can apply a variation of an argument due to Jennifer French and show that these spectral sequences agree with the unstable Adams spectral sequence.

From these equivalences we obtain information about filtration and differentials. Using these equivalences we construct the homological and cohomological Bockstein spectral sequences topologically.  We apply these spectral sequences to show that Hirzebruch genera can be lifted to $\mathcal{E}_\infty$--ring maps and that the forgetful functor from $\mathcal{E}_\infty$--algebras in $H\overline{\mathbb{F}}_p$--modules to $H_\infty$--algebras is neither full nor faithful.
 
\end{abstract}

\maketitle


\section{Introduction}
In \cite{JoN13} the author and Niles Johnson constructed the $T$--algebra spectral sequence which, as a special case, can be used to calculate the homotopy groups of the space of $\mathcal{E}_{\infty}$--algebra maps between two spectra. The purpose of this spectral sequence is to analyze the filtration between maps of $\mathcal{E}_{\infty}$--algebras `up to homotopy,' i.e., $H_{\infty}$--algebras, and actual $\mathcal{E}_{\infty}$--algebra maps.

Goerss and Hopkins also constructed a spectral sequence for computing $\mathcal{E}_{\infty}$--algebra maps \cite{GoH04} and below we will show that these two spectral sequences agree after restricting to algebras in $Hk$--modules for a field $k$. Although there is a direct argument for this equivalence due to Bousfield \cite[Thm.~6.2]{Bou03}, we will use a more general argument due to French \cite{Fre10}.

By taking function spectra one obtains a natural map:
\[\Top(X,Y_{\phat})\rightarrow \EinftyHfpbar(H\fpbar^{Y},H\fpbar^{X}),\]
where $Y_{\phat}$ is the $p$--nilpotent completion of $Y$ \cite{BoK72}. French showed that, under mild hypotheses, this map induces an isomorphism between the unstable Adams spectral sequence of Bousfield and Kan, which computes the homotopy groups of the source, and the Goerss--Hopkins spectral sequence which computes the homotopy groups of the target. This result has an obvious rational analogue and her argument can be applied to prove this case as well. These two results give detailed information about the embedding of rational and $p$--adic homotopy theory into categories of $\mathcal{E}_{\infty}$--ring spectra. Since we would like to use these results to analyze the latter categories we recall her argument in \Cref{sec:e2-model}.

We state these results and their immediate consequences in \Cref{sec:results}. In \Cref{sec:examples} we apply the $T$--algebra spectral sequence and these comparisons to show:
\begin{enumerate}
  \item Hirzebruch genera admit unique $\mathcal{E}_{\infty}$ lifts (\Cref{thm:genera}).
  \item Some of the results about $\mathcal{E}_\infty$ maps from \cite{JoN13} can be recovered using classical unstable homotopy theory (\Cref{sec:earlier}).
  \item If $X$ is a 2-connected spectrum, then every $H_{\infty}$--map from $\Sigma^{\infty}_+\Omega^{\infty}X$ to a $K(1)$--local $\mathcal{E}_{\infty}$--ring spectrum can be lifted to an $\mathcal{E}_{\infty}$--map (\Cref{prop:k-1}).
  \item In the category of $\mathcal{E}_{\infty}$--algebras in $H\fpbar$--modules, there is an uncountable family of homotopy classes of maps which induce the same $H_{\infty}$--algebra map (\Cref{thm:h-infty-vs-e-infty}.\ref{en:uncountable}).
  \item For each $n\geq 0$, there is an $H_{\infty}$--map in $H\fpbar$--modules for which the first $n$ obstructions to lifting to an $\mathcal{E}_{\infty}$--algebra map vanish, but which fails to lift to an $\mathcal{E}_{\infty}$--map. (\Cref{thm:h-infty-vs-e-infty}.\ref{en:d-k}).
  \item The Bockstein spectral sequences can be constructed as special examples of the $T$--algebra spectral sequence (\Cref{thm:bockstein}).
\end{enumerate}

This paper was written while the author was a visitor at the Max Planck Institute. He would like to thank the institute for their support and for providing a stimulating work environment. He would also like to thank the anonymous referee for their careful reading of this paper and their helpful comments.

\section{Comparisons}\label{sec:results}
\subsection{The $T$--algebra spectra sequence}
Let $T$ be a simplicial monad acting on a simplicial model category $\sC$. Under frequently satisfied hypotheses, the Eilenberg--Moore category of algebras $\sC_T$ inherits a simplicial model structure such that a morphism of $T$--algebras is a weak equivalence if and only if it is a weak equivalence in $\sC$.

Let $T=UF$ be the decomposition of $T$ into the left adjoint \[ F\co \sC\rightarrow \sC_T\] followed by the right adjoint $U$. Given $T$--algebras $X$ and $Y$, one can form the standard cotriple/bar resolution \[B_{\bul}(F,T,UX)=B_{\bul }X\rightarrow X\] of $X$ and then map this resolution into $Y$ to obtain a cosimplicial space $\sC_{T}(B_{\bul}X,Y)$. Now for any cosimplicial space one can construct the associated $\Tot$ spectral sequence of Bousfield and Kan \cite{Bok73a,Bou89}. One might hope that in our case this spectral sequence computes the homotopy groups of the derived mapping space $\sC_T(X,Y)$. This turns out to be too much to expect, even if $X$ and $Y$ are both cofibrant and fibrant.

It is important to note that the Bousfield--Kan spectral sequence is not a spectral sequence in the conventional sense. There are obstructions to obtaining a well-defined $E_n$ term for $n>1$, some of the terms are not abelian groups, and some of the differentials are relational. Moreover, there are generally obstructions to obtaining a base point in the totalization. Although the details are intricate, they are carefully treated in \cite{Bou89}, where it is shown that the spectral sequence simplifies when `certain Whitehead products vanish'. In these cases, the $E_2$ term is well-defined and every term except for $E_2^{0,0}$ is an abelian group. 

As explained in \cite[Thm.~4.5]{JoN13} this condition holds for all of the examples considered in this paper because, in each cosimplicial degree, we have an $H$--space. For this simplification, we do not require, nor is it usually true, that the cosimplicial structure maps be maps of $H$--spaces. In these nice cases, the obstructions to the existence of a base point lie in $E_2^{s,s-1}$ for $s\geq 2$. In our examples, we will either have an obvious base point or we will show that these groups vanish.

Using model categories, one can construct this spectral sequence by either first replacing $B_{\bul }X$ with a Reedy cofibrant simplicial diagram, or if $\sC_T$ is cofibrantly generated, taking a cofibrant replacement in the projective model structure. The new simplicial diagram has a geometric realization (resp.~homotopy colimit over $\Delta^{\op}$) which admits a well-behaved filtration by its skeleta (resp.~homotopy colimits over $\Delta_{\leq n}^{\op}$), which gives rise to a tower of fibrations. Taking homotopy groups at a fixed base point (which we will usually suppress), one obtains interweaving long exact sequences and an associated spectral sequence\footnote{Using \cite[Remark A.2.9.27]{Lur09} it is easy to see that the two approaches are equivalent if $\sC_T$ is combinatorial.}.

Unfortunately, even when this spectral sequence collapses and there are no $\lim^1$ issues it still may not converge to the homotopy groups of $\pi_* \sC_T(X,Y)$ because we have changed the homotopy type of the objects being studied. To obtain the correct convergence properties we assume that $X$ is resolvable \cite[Defn.~3.18]{JoN13}. This hypothesis is often satisfied and guarantees that $B_{\bul}X$ is already Reedy cofibrant. This is enough to construct a spectral sequence converging to $\pi_* \sC_T(|B_{\bul}X|,Y)$. Finally if $T$ commutes with geometric realizations or equivalently, the forgetful functor $U\colon \sC_T\rightarrow \sC$ preserves geometric realizations, then $|B_{\bul }X|\rightarrow X$ is a weak equivalence of $T$--algebras and we obtain the desired convergence. As shown in \cite[\S 3]{JoN13} each of these hypotheses will be satisfied in all of the examples we will consider.

\begin{prop}\label{prop:t-alg-ss}
  Let $k$ be a field. For any $\mathcal{E}_n$--algebras $X$ and $Y$ in $Hk$--modules there is a $T$--algebra spectral sequence converging to \[\pi_* \EnHk(X,Y).\] Moreover, the $E_2$ term of this spectral sequence is always defined and there are identifications
  \[E_{1}^{0,0} \cong \ho \HkMod(X,Y)\cong \kMod(\pi_* X,\pi_* Y)\]
  and
   \[ E_{2}^{0,0} \cong \Hnk(X,Y)\]
  such that the edge homomorphisms are the evident forgetful functors. Here $\ho\HkMod$ is equivalent to $\kMod$, the category of graded $k$--modules and $\Hnk$ is the category of $H_n$--algebras (or homotopy $\mathcal{E}_n$--algebras) in $\ho \HkMod$ \cite{BMMS86}.
\end{prop}
\begin{proof}
  The proof is exactly as in \cite[\S 4.3.3, \S 5.3]{JoN13} once we have shown that every $X$ is resolvable. To do this we need an $\mathcal{E}_n$--operad $\sO$ which is cofibrant and whose first component $\sO(1)$ is a point. The combinatorial $\mathcal{E}_n$--operad constructed by Berger \cite{Ber97} from Smith's filtration \cite{Smi89} satisfies the latter property. If we then apply the Boardman--Vogt $W_{\bul}$ construction to Berger's operad, we obtain an operad $\sO$ with both of the desired properties by \cite[Prop.~4.12.(c)]{JoN13}.
\end{proof}

When computing spaces of maps of $\mathcal{E}_1=\mathcal{A}_{\infty}$--algebras, $H_1$--algebras in $\ho \HkMod$ can be identified with associative graded $k$--algebras and the remainder of the $E_2$ term can be identified with the Quillen cohomology of the associative $k$--algebra $\pi_* X$ \cite[\S 5.3]{JoN13}. In positive cohomological degrees these cohomology groups can be identified with Hochschild cohomology groups \cite[Prop.~3.6]{Qui70}.

As described in \cite{May70} and \cite[\S III.1]{BMMS86}, when $n=\infty$ and $k$ is a field of characteristic $p$, $\pi_*X$ admits the structure of an graded commutative $k$--algebra with an action by the extended Dyer--Lashof algebra. In other words, in addition to the Bockstein, there are operations $Q^i$  acting on $\pi_*X$ for each integer $i$ which satisfy Adem and instability relations. The category of such algebras is an algebraic category in the sense of \cite{Qui70} and as a consequence one can construct a model structure on simplicial algebras over the extended Dyer--Lashof algebra and define cohomology with coefficients in an abelian group object in this category. As an example, if $t$ is positive, $\pi_* Y^{S^t}\cong \pi_{*}Y \oplus \pi_{*+t} Y$ is an abelian group object over $\pi_* Y$.

\begin{prop}[{\cite[Thm.~5.6]{JoN13}}]\label{prop:ident-e2}
  If we set $n=\infty$ in the spectral sequence of \Cref{prop:t-alg-ss} and let $f\colon \pi_* Y \rightarrow \pi_* X$ be a map of algebras over the extended Dyer--Lashof algebra, then
   \[E_2^{s,t}=HQ^s_{\pi_*Y}(\pi_* X;\pi_* Y^{S^t})\] 
   for $t>0$. Here the right hand side is the Quillen cohomology of the algebra $\pi_*X$ over the extended Dyer--Lashof algebra over $\pi_* Y$, via $f$, with coefficients in the abelian group object $\pi_* Y^{S^t}$.
\end{prop}

\begin{remark}
  Assuming that the functorial identifications of the homology of free $\mathcal{E}_n$--algebras in based spaces of \cite[\S III]{CLM76} can be extended to spectra and that one can form the associated algebraic category of $H_n$--algebras, then the remainder of the $E_2$ term in \Cref{prop:t-alg-ss} can be identified with appropriate Quillen cohomology groups just as in the $\mathcal{A}_{\infty}$ and $\mathcal{E}_{\infty}$ cases.
\end{remark}

\begin{remark}\label{rem:char-0}
  If $k$ is a field of characteristic zero, then there are no Dyer--Lashof operations and the homotopy theory of $\mathcal{E}_\infty$--algebras is equivalent to the homotopy theory of graded commutative differential graded algebras over $k$. Moreover, the $H_{\infty}$--algebras can be identified with graded commutative $k$--algebras. As in \Cref{prop:ident-e2}, we can then identify the $E_2$ term with the classical Andr\'e--Quillen cohomology groups \cite[\S 5.5]{JoN13}.
\end{remark}

\subsection{The Goerss--Hopkins Spectral Sequence}
The Goerss--Hopkins spectral sequence is another spectral sequence which can be used to calculate the space of $\sO$--algebra maps for a suitable operad $\sO$. This spectral sequence is generally constructed by resolving both the source and the target of the space of maps and forming the total complex. The resolution of the source uses a resolution model structure which will be described in some detail in \Cref{sec:e2-model}. For the target one uses an Adams resolution with respect to some nice homology theory $E$. This latter resolution is unnecessary when $E=Hk$ and we are working in the category of $Hk$--modules instead of spectra.

As described in \cite[\S 6]{GoH04}, the $E_2$ terms of the Goerss--Hopkins spectral sequences computing $\mathcal{E}_{\infty}$--algebra maps are abstractly isomorphic with the $E_2$ terms appearing in \Cref{prop:ident-e2,rem:char-0}.

\begin{prop}\label{prop:talgss-ghss}
  Let $k$ be a field and $X,Y\in \EnHk$. The $T$--algebra and Goerss--Hopkins spectral sequences abutting to $\pi_*\EnHk(X,Y)$ are isomorphic from the $E_2$ term on.
\end{prop}
\begin{proof}
  It suffices to show that the bar resolution used to construct the $T$-algebra spectral sequence is a resolution in the resolution model structure. This follows from the dual of \cite[Thm.~6.2]{Bou03}, but it will also follow from French's argument in \Cref{sec:e2-model}. 
\end{proof}

Now we recall one of the main results of French's thesis. 
\begin{thm}(\cite[Thm.~2.6.7] {Fre10})\label{thm:french}
  Suppose that $X$ and $Y$ are spaces of $\bF_p$--finite type and $Y$ is $p$--nilpotent. Then there is a natural isomorphism between the unstable Adam spectral sequence 
  \[E_{2,UA}^{s,t}\Longrightarrow \pi_*\Top(X,Y_{\hat{p}})\] 
  and the Goerss--Hopkins spectral sequence 
  \[E_{2,\mathcal{E}_{\infty}}^{s,t}\Longrightarrow \pi_*\EinftyHfpbar(H\fpbar^{Y},H\fpbar^X)\]
  from the $E_2$ page on.
\end{thm}

The two $E_2$ pages of these spectral sequences are ostensibly different and their identification depends on a crucial result of Mandell. In particular 
\[E_{2,\mathit{UA}}^{0,0}\cong \mathit{UA}(H^*(Y;\fp), H^*(X;\fp)),\] 
the maps of unstable algebras over the Steenrod algebra over $\fp$, while
\[E_{2,\mathcal{E}_\infty}^{0,0}\cong H_{\infty}(H\fpbar^{X},H\fpbar^{Y})\cong
\mathit{EUA}(H^*(Y;\fpbar), H^*(X;\fpbar))\] 
consists of the maps of unstable algebras over the \emph{extended Dyer--Lashof} algebra over $\fpbar$. The remainder of the $E_2$ term can be identified with the corresponding Quillen cohomology groups in these respective categories. These groups are also sometimes called non-abelian $\ext$ groups \cite{Mil84}\cite[\S 9]{Bou89}). Although we defer to \cite{Man01,Fre09,Fre10} for the details of this identification, we will recall a variation of French's argument for comparing these spectral sequences in \Cref{sec:e2-model} that is general enough to also prove:

\begin{thm}\label{thm:rational-french}
  Suppose that $X$ and $Y$ are spaces of $\bQ$--finite type and $Y$ is $\bQ$--nilpotent. Then there is a natural isomorphism between the unstable Adam spectral sequence converging to $\pi_*\Top(X,Y_{\bQ})$ and the Goerss--Hopkins spectral sequence converging to $\pi_*\EinftyHQ(H\bQ^{Y},H\bQ^X)$.
\end{thm}

Combining \Cref{prop:talgss-ghss} with \Cref{thm:french,thm:rational-french} we see that the unstable Adams spectral sequences can be recovered as special cases of the $T$--algebra spectral sequence. Since the latter has been studied off and on for the last 40 years, we have many powerful results that we can apply to analyze this filtration on $\mathcal{E}_{\infty}$--maps (see \Cref{sec:earlier} for some examples).

\section{Applications}\label{sec:examples}
\subsection{Hirzebruch Genera}\label{sec:genera}
Recall that $BGL_1S$ is the classifying space for stable spherical fibrations (see, for example, \cite{May77} where this is called $BF$). This is an infinite loop space and associated to any infinite loop map $G\rightarrow GL_1 S$ is an $\mathcal{E}_{\infty}$ Thom spectrum $MG$.  Standard examples include the $J$--homomorphisms from $\mathit{SO}$, $\mathit{Spin}$, $\mathit{String}$ and their complex analogues. The homotopy groups of these Thom spectra correspond to the bordism rings of the corresponding categories of manifolds.
 
Now a Hirzebruch genus is a graded commutative ring map
\[\overline{\phi}\colon \pi_* MG\rightarrow R_*\] 
whose target is a graded commutative $\bQ$--algebra (compare \cite[\S 19]{MiS74}). In geometric examples a genus is a function which assigns to each manifold an element in a $\mathbb{Q}$--module which is a cobordism invariant, takes disjoint unions to sums, and takes finite cartesian products to ordinary products.

\begin{thm}\label{thm:genera}
  There exists an $\mathcal{E}_{\infty}$--ring spectrum $R$ such that $\pi_* R\cong R_*$ and
  \[\pi_0 \mathcal{E}_{\infty}(MG,R)\cong \Comm(\pi_* MG, R_*).\]
  In other words, every Hirzebruch genus $\overline{\phi}$ can be lifted to a map of $\mathcal{E}_\infty$--ring spectra and this lift is unique up to homotopy through $\mathcal{E}_{\infty}$--ring maps.
\end{thm}
\begin{proof}
First note that we can apply the Eilenberg--MacLane functor to $R_*$ to obtain an $\mathcal{E}_{\infty}$--ring spectrum $HR$ with $\pi_*HR\cong R_*$. This is an arbitrary choice, but it will turn out the conclusion of the theorem does not depend on this choice.

Now any such genus factors 
\[\pi_* MG\rightarrow \pi_*MG\otimes \bQ\cong H_*(MG;\bQ)\rightarrow R_**\] 
and similarly there is a canonical weak equivalence of derived mapping spaces
\[\mathcal{E}_{\infty}(MG_{\bQ},HR)\xrightarrow{\sim} \mathcal{E}_{\infty}(MG,HR).\] 
Here $MG_{\bQ}$ is the Bousfield localization of $MG$ with respect to $H\bQ$.

We will now apply one of the above equivalent spectral sequences to compute
\[\pi_* \mathcal{E}_{\infty}(MG_{\bQ}, HR)\cong \pi_*\EinftyHQ(MG_{\bQ}, HR)\] 
and show that the obstructions to lifting $\overline{\phi}$ vanish.  Since the resulting spectra are rational we can use \cite[\S 5]{JoN13} to identify the $E_2$ term as follows:
\begin{align*}
   E_2^{0,0} &\cong  \Comm_{\bQ}(\pi_*MG_{\bQ},R_*)\cong
   \Comm(\pi_* MG,R_*)\\
   E_2^{s,t} & \cong H_{AQ,R_*}^{s}(\pi_*MG_{\bQ};\pi_*HR^{S^t}) \quad \text{ for } t>0.
\end{align*}
 
If $MG$ is $\bQ$--oriented, then we have a Thom isomorphism 
\[H_{*}(MG;\bQ)\cong H_{*}(BG;\bQ)\cong \pi_* MG_{\bQ}\] 
of graded-commutative $\bQ$--algebras. Now $H_{*}(BG;\bQ)$ is a connected bicommutative Hopf algebra over $\bQ$ and by the Milnor--Moore theorem \cite[\S 7]{MiM65} such Hopf algebras have free underlying graded commutative algebras.

So if $MG$ is $\bQ$--oriented, $\pi_*MG\otimes \bQ\cong \pi_* MG_{\bQ}$ is a free, and hence formal, $\bQ$--algebra and by \cite[Thm.~2.4.(ii)]{Qui70} the positive Andr\'e--Quillen cohomology groups for $\pi_* MG_{\bQ}$ vanish with any coefficients. By the identification of the $E_2=E_\infty$ page, the Hurewicz map \[\pi_0 \mathcal{E}_{\infty}(MG_{\bQ},HR)\rightarrow \Comm_{\bQ}(\pi_*MG\otimes \bQ,R)\cong \Comm(\pi_*MG,R)\]
is an isomorphism, so $\phi$ can be lifted uniquely up to homotopy. Moreover in this case there is a bijection between homotopy classes of $\mathcal{E}_{\infty}$--maps and $H_{\infty}$--ring maps.

Now by \cite[Prop.~IX.4.5]{LMS86}, Thom spectra are $\bQ$--oriented if and only if they are $\mathbb{Z}$--oriented, which happens if and only if the classifying map $G\rightarrow GL_1 S$ lifts to the connected cover $SL_1 S$. Moreover, if the map does not lift, then $\pi_0 MG$ is $\mathbb{Z}/2$, which implies $\pi_* MG$ is a $\mathbb{Z}/2$--algebra. In this case there are no maps to a $\bQ$--algebra, so the theorem vacuously holds.
\end{proof}

\begin{remark}
  Although the proof of \Cref{thm:genera} was not entirely formal, the only non-formal input was classical. The author expects that this result was surely known to experts by alternative methods which avoid the above spectral sequence arguments. In particular, the classical part of the argument above shows that $MG_\mathbb{Q}$ corresponds to a formal graded commutative $\bQ$--algebra via the Quillen equivalence between rational $\mathcal{E}_\infty$--ring spectra and graded commutative $\bQ$--algebras. It follows immediately that any such Hirzebruch genus can be lifted to any $\mathcal{E}_\infty$--map uniquely up to homotopy.

  \Cref{thm:genera} says that the problem of lifting genera to $\mathcal{E}_{\infty}$--maps is easy \emph{rationally}. In important cases, geometric arguments show that the Hirzebruch genus lifts to a non-rational ring. Integrally or even $p$--locally this problem is significantly harder and related to deeper mathematical questions (see, for example, \cite{AHR06}).
\end{remark}

\subsection{Examples from earlier work}\label{sec:earlier}
The equivalence between the unstable Adams spectral sequence and the $T$--algebra spectral sequence sheds light on \cite[Ex.~5.11, Ex.~5.14]{JoN13} where the $T$--algebra spectral sequence is used to calculate $\pi_* \mathcal{E}_{\infty}(H\bQ^{S^2},H\bQ^{S^3})$ and $\pi_* \mathcal{E}_{\infty}(H\bQ^{K(G,1)},H\bQ^{S^1})$. Here $G$ is the Heisenberg group. We can now identify these examples as calculations of $\pi_* \Top(S^3, S^2_{\bQ})$ and $\pi_* \Top(S^2, K(G,1)_{\bQ})$ respectively.

By evaluating at a base point of the sphere, we see that each of these spaces sits in a split fibration.  For example we have 
\[\Omega^2 K(G,1)_{\bQ}\rightarrow \Top(S^2, K(G,1)_{\bQ})\xrightarrow{ev} K(G,1)_{\bQ}.\]
These splittings decompose the spectral sequence into a product of spectral sequences computing the homotopy groups of \emph{based} spaces.

The importance of the first example was that it gave examples of non-trivial elements detected in positive filtration and hence were undetectable in the category of  $H_{\infty}$ ring spectra. These elements we can now identify as rational multiples of the Hopf map $S^3\rightarrow S^2$ and by convergence we see that these are permanent cycles. Moreover we can now bypass the direct computation in \cite{JoN13} to see that these elements must be in positive degree. Namely by \cite{BoK73} we know that Whitehead products always raise filtration degree and the Hopf map can be identified with the Whitehead product $[\mathrm{Id}_{S^2},\mathrm{Id}_{S^2}]$ up to a rational unit.

The second computation involving the Heisenberg group was important because it gave an example of a non-trivial differential (i.e., an obstruction). Now that we can see that we are computing the space of maps from $S^2$ into $K(G,1)$, we see that although the $E_2$ term is quite large, nearly all groups must be annihilated in the spectral sequence. The only non-base point elements to survive are in $E_2^{0,1}=\bQ\oplus \bQ$ and $E_2^{1,2}=\bQ$. These groups fit into the defining (non-split) exact sequence for the rational Heisenberg group. Again we can identify $E_2^{1,2}$ as generated by a Whitehead product, which in this instance is just the commutator of the two generators in $E_2^{0,1}$.

Finally we note that the differential constructed in \cite{JoN13} arises from the existence of a non-trivial Massey product in $H^*(K(G,1);\bQ)$. Using \cite{GuM74} or general staircase lemma arguments, one expects this Massey product to induce a differential in the Eilenberg--Moore spectral sequence converging to $H^*(\Omega K(G,1);\bQ)\cong \bQ[G]$. This forces a differential in the original spectral sequence by the results of \cite[\S 15]{BoK73} which show that the $E_2$ term of the rational unstable Adams spectral sequence \emph{embeds} (with a shift) into the $E_2$ term of the Eilenberg--Moore spectral sequence.

\subsection{$K(1)$--local examples}\label{sec:k-1}
In \Cref{thm:genera} we lifted genera to $\mathcal{E}_\infty$ maps using a simple trick: Although the spectrum $MG$ is not weakly equivalent to a free $\mathcal{E}_{\infty}$--ring spectrum, it is after rational localization. It follows that one can completely describe the space of $\mathcal{E}_{\infty}$--maps from $MG$ to any rational $\mathcal{E}_{\infty}$--ring spectrum.

A similar argument was used in \cite[Ex.~5.4]{JoN13} to analyze the space of $\mathcal{E}_{\infty}$ maps out of $\Sigma^{\infty}_{+}\Coker J$ to any $K(2)$--local $\mathcal{E}_{\infty}$--ring spectrum. That example depended both on \cite[Thm.~2.21]{Kuh06}, which gave criteria for suspension spectra to be $K(n)$--locally free, and joint work of the author with Nick Kuhn which showed that $\Coker J$ satisfies the required hypotheses when $n$ is $2$. When $n$ is $1$ we have the following:

\begin{prop}\label{prop:k-1}
  Suppose $X$ is a simply connected spectrum, $\pi_2 X$ is a finite torsion group, and $R$ is a $K(1)$--local $\mathcal{E}_{\infty}$--ring spectrum.
  Then the $T$--algebra spectral sequence converging to 
  \[ \pi_* \mathcal{E}_{\infty}(\Sigma^{\infty}_+ \Omega^{\infty} X, R) \] 
  collapses onto the 0--line. Moreover 
  \[\pi_* \mathcal{E}_{\infty}(\Sigma^{\infty}_+ \Omega^{\infty} X, R)\cong R^{-*}(X)\]
  and each $H_{\infty}$ map can be lifted uniquely, up to homotopy, to an $\mathcal{E}_{\infty}$--map.
\end{prop}
\begin{proof}
  By hypothesis $\Omega^{\infty}X$ satisfies the hypotheses of \cite[Thm.~2.21]{Kuh06}, so $\Sigma^{\infty}_+ \Omega^{\infty}X$ is $K(1)$--equivalent to the free $\mathcal{E}_{\infty}$--algebra on $X$. See \cite[Ex.~5.4]{JoN13} for further details.
\end{proof}

\subsection{Eilenberg--MacLane spaces and Bockstein spectral sequences}
Now let us use the unstable Adams spectral sequence, as a special case of the above spectral sequences to compute the homotopy groups of 
\[K(\bZ_{\widehat{p}},n)\simeq K(\bZ,n)_{\widehat{p}}\simeq \Top(*,K(\bZ,n)_{\widehat{p}})\simeq \EinftyHfpbar(H\fpbar^{K(\bZ,n)},H\fpbar) \]
for $n>0$. To base this spectral sequence we can take any base point of $K(\bZ,n)$. The calculation of this particular $E_2$ term is easy and well-known to experts, but since the author could not find it in the literature we include a brief sketch. To reduce clutter, unless otherwise stated, we will implicitly take cohomology with $\fp$--coefficients.

Now 
\[E_{2,\mathit{UA}}^{0,0}\cong \mathit{UA}(H^*(K(\bZ,n)),\fp)=\{\ast\}\]
which is the image of our base point. For positive $t$ we have
\[E_{2,\mathit{UA}}^{s,t}\cong H_{\mathit{UA},\bF_p}^{s}(H^*(K(\bZ,n)); H^*(S^t)).\]
While one could construct an explicit resolution of $H^*(K(\bZ,n))$ in the category of simplicial unstable algebras over the Steenrod algebra following the method of \cite[Ex.~5.11]{JoN13}, we will instead use the composite functor spectral sequence of \cite[Thm.~2.5]{Mil84} which converges to our desired $E_2$ term. Since we will only use this spectral sequence here and it collapses at $E_2$ we will just say that the two derived functors being used are those arising from the derived indecomposables of $H^*(K(\bZ,n))$, as an $\fp$--algebra, and the higher $\Ext$ groups in the category of unstable \emph{modules} over the Steenrod algebra (see \cite{Sch94} for a detailed study of this category).

Now by classical computations of Cartan and Serre, $H^*(K(\bZ,n))$ is a free graded commutative commutative algebra on a module $Q$. Here $Q$ is the free module on those admissible monomials over the Steenrod algebra acting on an element in degree $n$, of excess no greater than $n$, modulo those admissible monomials beginning with a Bockstein. This is $\beta$ for odd primes and $\mathit{Sq}^1$ at the prime $2$. This implies the higher Andr\'e--Quillen homology groups vanish and the composite functor spectral sequence collapses to yield 
\[ E_{2,UA}^{s,t}\cong \Ext_{UM}^{s}(\Sigma^{-1}Q, \Sigma^{t-1}\fp).\]

These $\Ext$ groups are defined with respect to a cohomological grading, so we will temporarily switch our grading convention. Let $\sF(i)$ denote the free unstable module on a generator of degree $i$ \cite{Sch94}. Now to calculate these $\Ext$ groups we will construct the following `periodic' resolution of $\Sigma^{-1} Q$:
\[ \cdots \xrightarrow{\beta} \sF(n+1)\xrightarrow{\beta} \sF(n)\xrightarrow{\beta}
\sF(n-1)\xrightarrow{i} \Sigma^{-1} Q\] 
where $\beta$ is the first Bockstein operation and $i$ is the canonical map hitting the desuspended fundamental class. It is easy to see that the relation $\beta^2=0$ and the definition of $\sF(i)$ imply this sequence is exact.

This immediately gives us the $E_2$ term shown in \Cref{fig:kzn}. Clearly the spectral sequence collapses and by convergence, every possible group extension is non-trivial.

\begin{figure}
  \begin{center}
  \pgfimage{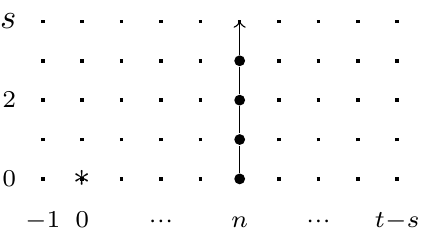}\quad\quad
  \pgfimage{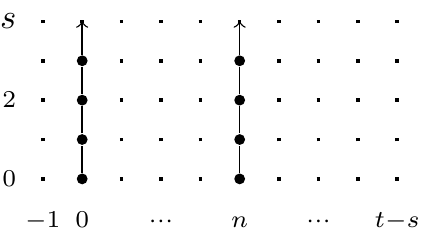}
  \end{center}
  \caption{\label{fig:kzn}
    $E_2=E_{\infty}$ pages for
    $\pi_*\EinftyHfpbar(H\fpbar^{K(\bZ,n)},H\fpbar)$
    and 
    $\pi_*\EinftyHfpbar(H\fpbar^{K(\bZ,n)},H\fpbar^{S^n})$ respectively.}
\end{figure}

To determine the behavior of the spectral sequences above converging to $\pi_*K(\bZ/p^k,n)$, we use the fibration sequence
\[K(\bZ/p^k, n)\rightarrow K(\bZ,n+1)\xrightarrow{p^k} K(\bZ,n+1) \]
and the corresponding Eilenberg--Moore spectral sequence:
\begin{equation*}
\begin{split} \Tor^{s,t}_{H^*(K(\bZ,n+1))}(H^*(K(\bZ,n+1)), \bF_p)\cong \Tor^{s,t}_{H^*(K(\bZ,n+1))}(\fp, \fp)\otimes H^*(K(\bZ,n+1))\\ \Longrightarrow H^{t-s}(K(\bZ/p^k,n)).
\end{split}\end{equation*}
Comparing with the classical computations of Serre and Cartan we see that, for all $k\geq 1$, this spectral sequence collapses.

If we assign the indecomposable generators of $H^*(K(\bZ,n))$ filtration degree 1 and the other generators degree 0, then the $E_2=E_{\infty}$ term of the Eilenberg--Moore spectral sequence is isomorphic to the associated graded of $H^*(K(\bZ,n))\otimes H^*(K(\bZ,n+1))$. Since these algebras are all free as graded commutative $\fp$--algebras there are no additive or multiplicative extensions. However when $k=1$ the Bockstein connects the fundamental classes of these two algebras yielding a free unstable algebra over the Steenrod algebra. In general these algebras will be connected by a $k$th order Bockstein.

So for $k>1$ we see that \[H^*(K(\bZ/p^k,n))\cong H^*(K(\bZ,n))\otimes H^*(K(\bZ,n+1))\] as unstable algebras over the Steenrod algebra. Andr\'e--Quillen cohomology takes coproducts of algebras over the Steenrod algebra, which are represented by tensor products, to products. It follows that the $E_2$ term of the spectral sequence computing $\pi_*K (\bZ/p^k,n)$ is concentrated in the $t-s=n$ and $n+1$ columns. Convergence of this spectral sequence forces the pattern of differentials in \Cref{fig:kzmodp}:
\begin{figure}
\begin{center}
  \pgfimage{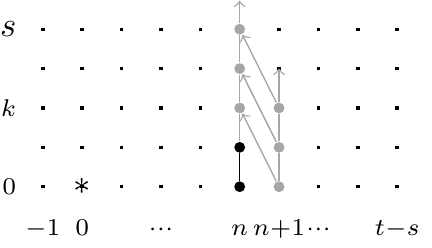}\quad\quad
  \pgfimage{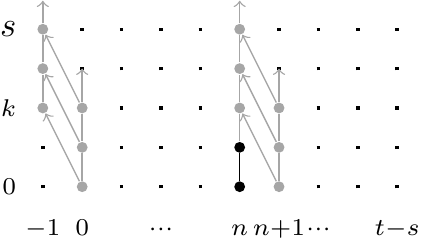}
  \end{center}
  \caption{\label{fig:kzmodp}For $k>1$, the
    $E_{k}$/$E_{\infty}$ pages for
    $\pi_*\EinftyHfpbar(H\fpbar^{K(\bZ/p^k,n)},H\fpbar)$ and
    $\pi_*\EinftyHfpbar(H\fpbar^{K(\bZ/p^k,n)},H\fpbar^{S^{n+1}})$
    respectively.}
\end{figure}

Just as for the examples in \Cref{sec:earlier}, we can `loop' these examples down to see that the $E_{\infty}^{s,s}$ terms in the spectral sequence computing $\pi_* \EinftyHfpbar(H\fpbar^{K(\bZ/p^k,n)},H\fpbar^{S^n})$ are each non-trivial for $0\leq s<k$. While in the spectral sequence computing $\pi_* \EinftyHfpbar(H\fpbar^{K(\bZ,n)},H\fpbar^{S^n})$, we see that $E_\infty^{s,s}$ is non-trivial for all $s\geq 0$ (\Cref{fig:kzn}). We also see that there are elements in $E_k^{0,0}$ in the spectral sequences computing
\[\pi_* \EinftyHfpbar(H\fpbar^{K(\bZ/p^k,n)}, H\fpbar^{S^{n+1}})\] 
which support a non-trivial $d_k$ (\Cref{fig:kzmodp}). The last two observations yield the following:
\begin{thm}\label{thm:h-infty-vs-e-infty}\ 
  \begin{enumerate}
    \item \label{en:uncountable} In the category of $\mathcal{E}_{\infty}$--algebras in $H\fpbar$--modules there is an uncountable family of homotopy classes of $\mathcal{E}_{\infty}$--maps inducing the same $H_{\infty}$--map.
    
    \item \label{en:d-k} Moreover for each $n\geq 0$, there is an $H_{\infty}$--map for which the first $n$--obstructions to lifting to an $\mathcal{E}_{\infty}$--algebra map in $H\fpbar$--modules vanish but which fails to lift to an $\mathcal{E}_{\infty}$--map.
  \end{enumerate}
\end{thm}

 When $X$ is a degree-wise finite product of Eilenberg--MacLane spaces of finitely generated abelian groups, the calculations above dictate the behavior of the spectral sequences computing $\Top(*, X_{p})$. Namely, the $d_k$ differentials are always induced by non-trivial $k$th order Bockstein operations. In particular this holds if we set $X$ to be the reduced infinite symmetric power of a connected space $Y$ of finite type.  In this case, by the Dold--Thom theorem, 
\[X\simeq \prod_{i\geq 1} K(\widetilde{H}_{i}(Y;\mathbb{Z}),i).\]

Similarly if $Y$ is an $n$--dimensional finite CW-complex, then $X=\Top(Y,K(\bZ,n+1))$ is a connected space weakly equivalent to a product of Eilenberg--MacLane spaces such that $\pi_* X \cong H^{n+1-*}(Y;\bZ)$.

Applying our analysis above to these examples we obtain:

\begin{thm}\label{thm:bockstein}
  Let $Y$ be a connected space of finite type and $Z$ an $n$--dimensional finite CW-complex. Then the $T$--algebra spectral sequences converging to
  \[
    \pi_* \EinftyHfpbar(H\fpbar^{\Sym^{\infty}_* Y},H\fpbar)\cong
    \widetilde{H}_*(Y;\bZ_{p})
  \]
  and 
  \[
    \pi_* \EinftyHfpbar(H\fpbar^{\Top(Z,K(\bZ,n+1))},H\fpbar)\cong
    H^{n+1-*}(Z;\bZ_{p})
  \]
  agree with the corresponding Bockstein spectral sequences from the
  $E_2$ page on.
\end{thm} 

\section{The $E_2$--model structure and a comparison
result}\label{sec:e2-model}
Let $\sP\subset \ho\HkMod$ be the class of spectra containing all suspensions and desuspensions of $Hk$ where $k$ is a field. Following \cite{Bou03} we will say a map $f\colon C\rightarrow D$ is $\sP$--epi if the induced map
\[ \ho\HkMod(P,C)\rightarrow \ho\HkMod(P,D)\] 
is surjective for all $P\in \sP$, or equivalently if the induced map on homotopy groups is a surjection. An $Hk$--module $P$ is $\sP$--projective if 
\[\ho\HkMod( P,C)\rightarrow \ho\HkMod(P,D) \]
is surjective for all $\sP$--epi maps $C\rightarrow D$. A morphism $A\rightarrow B$ is called a $\sP$--projective cofibration if it has the left lifting property against all $\sP$--epi maps. Clearly $\hkmod$ has enough $\sP$--projectives in the sense that every $X\in \hkmod$ is the target of a $\sP$--epi map from a $\sP$--projective.

\begin{defn}\
\begin{enumerate}
  \item A map $f\colon A\rightarrow B$ of simplicial $Hk$--modules is a $\sP$ or resolution weak equivalence if the induced map $\pi_* A \rightarrow \pi_* B$ is an equivalence of graded simplicial abelian groups. Equivalently by the Dold--Kan correspondence, $f$ is a $\sP$--weak equivalence if, for each $n$, the induced map between normalized chain complexes $N(\pi_n A)\rightarrow N(\pi_n B)$ is a quasi-isomorphism.

  \item A map $f$ is a $\sP$--fibration if it is a Reedy fibration and for all $n$, $\pi_n f$ is a fibration of simplicial $k$--modules or equivalently, $N(\pi_n A)\rightarrow N(\pi_n B)$ is surjective in positive degrees.

  \item Finally we will say that $f$ is a $\sP$--cofibration if each of the induced latching maps \[ A_n \coprod_{L_n A} L_n B\rightarrow B_n \] is a $\sP$--projective cofibration.
\end{enumerate}
\end{defn}
By \cite[Thm.~3.3]{Bou03}, \cite[p.~19]{GoH04} the $\sP$--fibrations, $\sP$--cofibrations, and $\sP$--weak equivalences form a simplicial cofibrantly generated model structure on $\shkmod$ which we will call the resolution model structure.

Now the forgetful functor 
\[U\colon \EinftyHk\rightarrow \hkmod\]
is a simplicial monadic right Quillen functor. Here a map $f$ of $\mathcal{E}_{\infty}$--algebras is a weak equivalence/fibration if and only if $Uf$ is a weak equivalence/fibration. If we let $s\sC=\sC^{\Delta^{\op}}$ denote the category of simplicial objects in $\sC$. The forgetful functor above extends to a simplicial monadic right Quillen functor
\[U\colon \sEinftyHk \rightarrow \shkmod\] 
where $\shkmod$ is equipped with the resolution model structure above \cite[Thm.~3.12]{GoH04}.

Indeed since these are simplicial categories, we have a natural path object and by standard arguments with cofibrantly generated model categories (e.g., \cite[Thm.~3.1]{Sch01}), the only essential point to check is the existence of a fibrant replacement functor. For this purpose one can take the fibrant replacement in the Reedy model structure.

We can regard a given $\mathcal{E}_{\infty}$--algebra $R$ as a constant simplicial object and take its cofibrant replacement $R_{\bul}$ in the resolution model structure. As in \cite[Thm.~3.12]{GoH04}, by taking a level-wise weakly equivalent diagram we can further assume, by a process they call `subdivision', that the underlying degeneracy diagram of $R_\bul$ is obtained by a left Kan extension from a level-wise free diagram on a discrete category\footnote{As explained in \cite{GoS07}, in the context of simplicial commutative algebras, the subdivided complex can be obtained by taking the canonical cellular replacement arising from the cofibrantly generated resolution model structure on $\sEinftyHk$.}. This immediately implies the diagram is Reedy cofibrant which in turn implies that the natural maps $|\sk_n R_{\bul}|\rightarrow |\sk_{n+1}R_{\bul}|$ are cofibrations of $\mathcal{E}_{\infty}$ algebras and that $|R_{\bul}|$ is cofibrant.

By mapping into a fibrant $\mathcal{E}_{\infty}$--ring spectrum $Y$, we now obtain a Reedy fibrant cosimplicial space $\EinftyHk(R_{\bul},Y)$, from which we obtain a corresponding Bousfield--Kan spectral sequence \cite{Bou89}: 
\[ \pi^s\pi_t \EinftyHk(R_{\bul},Y)\Longrightarrow \pi_{t-s}\EinftyHk(|R_{\bul}|,Y).\]
The $E_2$ term exists because $R_\bul$ is a simplicial resolution by $H$--cogroup objects and, as a consequence, the cosimplicial mapping space is an $H$--space in each degree. By the results of \cite[\S 3]{GoH04} we see that $|R_{\bul}|\xrightarrow{\sim}R$, so this spectral sequence converges to the desired target.
 
Now associated to any operad $\sO$ is a monad $T_{\sO}$ on $\HkMod$ such that the categories of algebras over the operad $\sO$ is equivalent to the category of algebras over the monad $T_{\sO}$. As a functor $T_{\sO}$ takes a module to the free $\sO$--algebra on that module. When the category of $\sO$--algebras admits a model structure induced from $\HkMod$ and we have a functorial identification of the free algebras, the monad $T_{\sO}$ descends to a monad $T_{H_*\sO}$ on $\ho\HkMod$. Since $k$ is a field, the latter category is equivalent to the category of graded $k$--vector spaces. In our case, $\sO=\mathcal{E}_{\infty}$, and the category of $T_{H_*\sO}$--algebras is the category of algebras over the extended Dyer--Lashof algebra that satisfy admissibility relations. Rationally this is just the category of graded commutative $\bQ$--algebras.

Now $\pi_*R$ is naturally an $H_*\sO$--algebra which we can regard as a constant simplicial object in simplicial $H_*\sO$--algebras. Using the machinery of \cite{Qui69} the category of such objects $\sHoOAlg$, inherits the structure of a simplicial model category, where the weak equivalences and fibrations are those of the underlying simplicial sets.

By construction of our resolution, $R_{\bul}$ has a free underlying degeneracy diagram, so $\pi_*R_{\bul}$ has a free underlying degeneracy diagram. It follows from \cite[Prop.~4.1]{GoH04}, that $\pi_*R_{\bul}\rightarrow \pi_*R$ is a cofibrant replacement in $\sHoOAlg$. Moreover taking homotopy groups level-wise induces an isomorphism
\begin{equation}\label{eqn:homotopy-homs}
  \ho\EinftyHk(R_{\bul},Y)\xrightarrow{\simeq} \HoOAlg(\pi_*R_{\bul},\pi_*Y).
\end{equation}
This is the $\pi_0$--analogue of the following equivalence for positive
$t$:
\[ 
  \pi_t (\EinftyHk(R_{\bul},Y),f)\cong \ho(\EinftyHk)_{Y}(R_{\bul},Y^{S^t})
  \xrightarrow{\simeq} \HoOAlg_{\pi_* Y}(\pi_*R_{\bul},\pi_*Y^{S^t}).
\]
Now $\pi_* Y^{S^t}\cong \pi_* Y \oplus \pi_{*+t} Y$ as graded modules, and as algebras it is a square zero extension.


Since such square zero extensions represent abelian group objects in the overcategory $\HoOAlg_{\pi_* Y}$,
\[\HoOAlg_{\pi_* Y}(\pi_*R_{\bul},\pi_*Y^{S^t})\]
is a cosimplicial abelian group, which by Dold--Kan corresponds to a cochain complex. Since $\pi_*R_{\bul}\rightarrow \pi_*R$ is a cofibrant replacement in $\sHoOAlg$, the $s$th cohomology group of this complex is, by definition, the $s$th Quillen cohomology group $H^s_{\HoOAlg}(\pi_* R; \pi_*Y^{S^t})$. Similarly \eqref{eqn:homotopy-homs} defines a cosimplicial set. Applying $\pi^0$ or equivalently, taking the equalizer of $d^0$ and $d^1$, we obtain
\[
\pi^0\pi_0 \EinftyHk(R_{\bul},Y)\xrightarrow{\simeq} \HoOAlg(\pi_*R,\pi_*Y).
\]
These observations combine to give a complete description of the
$E_2$ term of the Goerss--Hopkins spectral sequence. We
will call the groups $\pi^s \pi_t \EinftyHk(R_{\bul},Y)$ the
$(s,t)$--derived functors of $R$ (with respect to $Y$).

Now suppose we take an alternative resolution $f\colon\widetilde{R}_{\bul}\rightarrow R$, such that $f$ is a resolution weak equivalence and $\EinftyHk(\widetilde{R}_{\bul},Y)$ is a Reedy fibrant diagram of $H$--spaces. We are interested in the case when $\widetilde{R}_\bul$ is \emph{not} cofibrant in the resolution model structure. Associated to these resolutions we have two spectral sequences
\begin{align*}
  \pi^s \pi_t \EinftyHk(R_{\bul},Y) \Longrightarrow &
  \pi_{t-s}\EinftyHk(|R_{\bul}|, Y) \textrm{ and }\\
  \pi^s \pi_t \EinftyHk(\widetilde{R}_{\bul},Y) \Longrightarrow &
  \pi_{t-s}\EinftyHk(|\widetilde{R}_{\bul}|, Y).
\end{align*}
It is natural to ask under which hypotheses the spectral sequences agree.  

\begin{thm}\label{thm:flat-resolutions}
  Suppose $\widetilde{R}_{\bul}\rightarrow R$ is a resolution weak equivalence satisfying the above conditions. Let $P_{\bul}\widetilde{R}_n\rightarrow \widetilde{R}_n$ be a cofibrant replacement in the resolution model structure. Suppose that for each $n$ and pair $(s,t)$ of positive integers, the cohomotopy groups $\pi^s\pi_t \EinftyHk(P_{\bul}\widetilde{R}_n, Y),f)$ vanish for all potential base points $f$. Then there is a natural isomorphism between the totalization spectral sequences associated to applying $\EinftyHk(-,Y)$ to $\widetilde{R}_{\bul}$ and a cofibrant resolution of $R$. Moreover this isomorphism is induced by a weak equivalence between the totalizations.
\end{thm}
\begin{proof}
Following ideas from homological algebra, to compare these spectral sequences we will construct a biresolution which maps to both.  For this we consider the Reedy model structure on simplicial objects in the resolution model structure on $\sEinftyHk$, which we call the Reedy--resolution model structure.\footnote{French uses an analogously defined injective--resolution model structure. By using the Reedy variant we can use any of the categories of spectra considered in \cite{GoH04} and not just the combinatorial models. We can also directly apply the results of \cite{Bou03}.} The first/horizontal simplicial direction we will refer to as the Reedy direction and the second/vertical direction as the resolution direction.

Regard $\widetilde{R}_{\bul}$ as a constant bisimplicial algebra in the
resolution direction. By applying cofibrant replacement and subdivision
to the map $\widetilde{R}_{\bul}\rightarrow R$ we obtain a (weak) map of
bisimplicial algebras
\[P_{\bul}\widetilde{R}_{\bul}\rightarrow R_{\bul},\] 
which is level-wise free and cofibrant. 

So we have a zig-zag of diagrams 
\begin{equation}\label{eqn:bico}
 \widetilde{R}_{\bul} \leftarrow P_{\bul}\widetilde{R}_{\bul} \rightarrow R_{\bul}.
\end{equation} 
After taking diagonalizations and mapping into $Y$ we obtain maps of cosimplicial spaces:
\begin{equation}\label{eqn:comparison}
  \EinftyHk(\widetilde{R}_{\bul},Y) \rightarrow \EinftyHk(d(P_{\bul}\widetilde{R}_{\bul}),Y) \leftarrow \EinftyHk(R_{\bul},Y). 
\end{equation}
We want to see that these maps induce isomorphisms of spectral sequences from the $E_2$ page on for each choice of potential base point $f$. In particular, we want to see that the $E_2^{0,0}$ terms, which are the sets of potential base points and do not depend on any fixed choice of base point, are isomorphic. Once we have constructed these equivalences of spectral sequences, we can apply \cite[p.~73]{Bou89} to see that these spectral sequence isomorphisms are induced by weak equivalences between their totalizations.

Now the left arrow in \eqref{eqn:bico} is a Reedy--resolution weak equivalence by construction and by the dual of \cite[Lemma 6.9]{Bou03} the induced map on diagonalizations is a resolution weak equivalence. The diagonal functor from the Reedy--resolution model structure to the resolution model structure also preserves cofibrations by \cite[Lemma 6.10]{Bou03}. So after diagonalizing \eqref{eqn:bico} the right two objects become resolution cofibrant. In each case, applying subdivision if necessary and using our assumption on $\widetilde{R}_\bul$, we see that \eqref{eqn:comparison} is a diagram of Reedy fibrant cosimplicial spaces which are $H$--spaces in each degree. As a consequence, we see that the associated Bousfield--Kan spectral sequences have well-defined $E_2$ terms and the maps in \eqref{eqn:comparison} induce morphisms of spectral sequences once we have chosen a compatible system of base points.

Now the induced map $d(P_{\bul}\widetilde{R}_{\bul}) \rightarrow R_{\bul}$ is a map over resolution weak equivalences to $R$ and hence a resolution weak equivalence. Since these are both cofibrant in the resolution model structure and the Goerss--Hopkins spectral sequence is invariant under choice of cofibrant resolution \cite[p.~24]{GoH04}, this map induces a morphism of spectral sequences which is an isomorphism at $E_2$. 

Now we will want to use the `totalization' spectral sequence (see \Cref{rem:totalization-ss})
\begin{equation}\label{eqn:tot-ss}
  \pi^v\pi^h \pi_t \EinftyHk(P_{\bul}\widetilde{R}_{\bul}, Y) \Longrightarrow
\pi^{v+h}\pi_t \EinftyHk(d(P_{\bul}\widetilde{R}_{\bul}), Y) 
\end{equation} 
with an edge homomorphism
\[ \pi^v \pi_t \EinftyHk(\widetilde{R}_{\bul}, Y) \hookrightarrow \pi^v\pi^h \pi_t \EinftyHk(P_{\bul}\widetilde{R}_{\bul}, Y) \]
given by inclusion of the 0th horizontal cohomotopy groups.  Here we finally apply our assumption that $\widetilde{R}_n$ has no higher derived functors. This implies that the above spectral sequence collapses to the 0th column. So the edge homomorphism gives the desired equivalence
\[\pi^v \pi_t \EinftyHk(\widetilde{R}_{\bul}, Y) \cong \pi^v \pi_t \EinftyHk(d(P_{\bul}\widetilde{R}_{\bul}), Y)\]
of $E_2$ terms. 
\end{proof}

\begin{remark}
  Although we stated \Cref{thm:flat-resolutions} in terms of $\mathcal{E}_{\infty}$--algebras it should be clear that the arguments above are general and can be applied to other Bousfield--Kan spectral sequences constructed from Bousfield's resolution model structure such as those in \cite[\S 5.8]{Bou03}.
\end{remark}

\begin{remark}\label{rem:totalization-ss}
Although the totalization spectral sequence has appeared previously in the literature, the author does not know of any formal account. The construction is standard, but there are a couple of subtleties. With the exception of the part containing $\pi_0$ this is a graded sequence of spectral sequences.

First, for each positive $t$ we will construct a different totalization spectral sequence for the bicosimplicial group $\pi_t\EinftyHk(P_{\bul}\widetilde{R}_{\bul}, Y)$. Since in each bidegree we have an $H$--space, these are bicosimplicial \emph{abelian} groups for each choice of base point in tridegree (0,0,0). So we can apply the Dold--Kan correspondence to rephrase the construction in terms of bicochain complexes and use the standard spectral sequence converging to the cohomology of the total complex. 

There is no accepted definition of the higher cohomotopy of a cosimplicial set. So for $t=0$, there is only a single defined term on either side of \eqref{eqn:tot-ss} and there is no room for differentials. So this `spectral sequence' is just the claim 
\[\pi^0\pi^0 \pi_0 \EinftyHk(P_{\bul}\widetilde{R}_{\bul}, Y) \cong \pi^{0}\pi_0 \EinftyHk(d(P_{\bul}\widetilde{R}_{\bul}), Y). \]
This is a general fact about bicosimplicial sets, which we will now verify.

The lower terms of a bicosimplicial set fit into the following diagram:
  \begin{center}
  \pgfimage{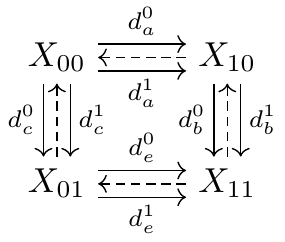}
  \end{center}
The dashed arrows are the codegeneracies $s^0_?$. The claim we would like to verify is that the intersection of the equalizers of the pairs $(d^0_a,d^1_a)$ and $(d^0_c,d^1_c)$ is equal to the equalizer of $(d^0_e d^0_c,d^1_e d^1_c)$.

Note that if $d^0_a(x)=d^1_a(x)$ then 
\[d^0_e(d^0_c(x))= d^0_b(d^0_a(x))=d^0_b(d^1_a(x))=d^1_e(d^0_c(x)).\]
So the intersection of the equalizers of $(d^0_a,d^1_a)$ and $(d^0_c,d^1_c)$ is in the equalizer of $(d^0_e d^0_c,d^1_e d^1_c)$. To see the other containment suppose $d^0_e d^0_c(x)=d^1_e d^1_c(x)$. By applying $s^0_e$ to both sides and using the cosimplicial identities we see that $d^0_c(x)=d^1_c(x)$. The commutativity of the diagram shows that $d^0_e d^0_c(x)=d^1_e d^1_c(x)$ implies $d^0_b d^0_a(x)=d^1_b d^1_a(x)$. Applying the codegeneracy $s^0_b$ then shows $d^0_a(x)=d^1_a(x)$.

\end{remark}

In homological algebra one can use flat resolutions of a module to compute $\Tor$ groups. We can regard \Cref{thm:flat-resolutions} as a non-abelian analogue of this result. This is also a generalization of \cite[Thm.~6.2]{Bou03}. 

To apply \Cref{thm:flat-resolutions} we need to find Reedy cofibrant resolutions of homotopy comonoids whose terms have no higher derived functors. It is immediate from the definitions that the bar resolutions for algebras over operads in $\HkMod$ considered in \cite{JoN13} are level-wise homotopy comonoids and projective in the resolution model structure. It follows that the algebras $B_n R$ appearing in this construction have no higher derived functors. By \cite[Prop.~4.22]{JoN13} these resolutions are Reedy cofibrant, so the argument above gives an alternate proof of \Cref{prop:talgss-ghss}.

Now the unstable Adams spectral sequence of Bousfield--Kan arises from mapping a space $X$ into the cobar cosimplicial resolution $\bF_{p}^{\bul+1}Y$. This resolution is constructed from the monad $\bF_{p}$ which takes a simplicial set to the free simplicial $\fp$--vector space on that simplicial set. To compare the unstable Adams spectral sequence of Bousfield--Kan to the Goerss--Hopkins spectral sequence the first step is to show the following:
\begin{thm}\label{thm:vanishing}\
  \begin{enumerate}
    \item \label{it:rational-case} If $X$ is a connected space of $\bQ$--finite type, then the $(s,t)$--derived functors of $H\bQ^{\bQ X}$ with respect to any rational $\mathcal{E}_{\infty}$--algebra $Y$ vanish for $s$ positive.
    \item \label{it:char-p-case} If $X$ is a connected space  of $\bF_p$--finite type, then the $(s,t)$--derived functors of $H\fpbar^{\bF_p X}$ with respect to any $\mathcal{E}_{\infty}$--algebra of the form $H\fpbar^{Y}$ vanish for $s$ positive. Moreover the $(0,t)$--derived functors can be identified with $E_{2,UA}^{0,t}$.
  \end{enumerate}
\end{thm}
\begin{proof}
   By the identification in \cite[Ex.~5.9]{JoN13}, to prove \eqref{it:rational-case} we have to show the higher Andr\'e-Quillen cohomology groups of $H^*(\bQ X;\bQ)$ vanish as graded commutative $\bQ$--algebra. These groups will vanish if the cohomology ring is free as a graded commutative algebra. Since $X$ is of finite $\bQ$--type, $\bQ X$ is of finite type and weakly equivalent to a product of rational Eilenberg--MacLane spaces. Using K\"unneth isomorphisms it suffices to show the $H^*(K(\bQ,n);\bQ)$ is free as a graded commutative algebra for $n>0$. This is both easy and classical.

   The proof of \eqref{it:char-p-case} depends on Mandell's explicit resolution \cite[\S 6]{Man01} of $\mathcal{E}_{\infty}$--algebras of this form and appears in \cite[\S 2.3]{Fre10}.
\end{proof}

The functor $\Top^{\op}\rightarrow \EinftyHfpbar$ which sends a space $X$ to the function spectrum $H\fpbar^{X}$ defines a map of cosimplicial spaces 
\[ \Top(X,\bF_{p}^{\bul+1}Y)\rightarrow \EinftyHfpbar(H\fpbar^{\bF_{p}^{\bul+1}Y},H\fpbar^{X}). \]
After replacing $H\fpbar^{\bF_{p}^{\bul+1}Y}$ with a Reedy cofibrant replacement, we obtain a map from the unstable Adams spectral sequence to the latter spectral sequence which, by \Cref{prop:talgss-ghss,thm:flat-resolutions,thm:vanishing}, agrees with the Goerss--Hopkins and $T$--algebra spectral sequences.

Using \Cref{thm:vanishing}.\ref{it:char-p-case} we can identify the $E_2$ term of the latter spectral sequence with $E_{2,\mathit{UA}}$ as described above \cite[\S 2.3]{Fre10}. Although the $E_2$ term of the classical unstable Adams spectral sequence consists of $\Ext$--groups of unstable \emph{coalgebras} over the Steenrod algebra, the map between the spectral sequences is a duality map, and our algebras are of finite type, so we obtain \Cref{thm:french}.

\bibliographystyle{gtart}
\bibliography{biblio}

\end{document}